\documentclass[11pt]{amsart}

\usepackage{amsthm}
\usepackage{amsmath}
\usepackage{amssymb}
\usepackage{color}

\newcommand{\tr}{\mathop\mathrm{tr}}

\newcommand{\spn}{\mathop\mathrm{span}}
\newcommand{\coh}{\mathrm{coh}}

\newcommand{\FUNTF}{\mathrm{FUNTF}}

\newtheorem{thm}{Theorem}[section]
\newtheorem{theorem}[thm]{Theorem}

\newtheorem{corollary}[thm]{Corollary}
\newtheorem{problem}[thm]{Problem}

\newtheorem{lemma}[thm]{Lemma}
\newtheorem{proposition}[thm]{Proposition}

\newtheorem{definition}[thm]{Definition}
\theoremstyle{remark}
\newtheorem{remark}[thm]{Remark}

\newtheorem{ex}[thm]{Example}

\newcommand{\RR}{\mathbb R}
\newcommand{\NN}{\mathbb N}
\newcommand{\CC}{\mathbb C}

\newcommand{\ip}[2]{\left\langle#1,#2\right\rangle}

\newcommand{\absip}[2]{\left| \left\langle#1,#2\right\rangle \right|}

\begin{document}

\title{The core of a Grassmannian frame}
\author[Casazza, Campbell, Tran]{Peter G. Casazza, Ian Campbell, Tin T. Tran}
\address{Campbell/Casazza: Department of Mathematics, University
of Missouri, Columbia, MO 65211-4100; \\
Tran: Department of Mathematics and Statistics 11200 SW 8th Street, DM 430 Miami, FL 33199 }

\thanks{The authors were supported by
 NSF DMS 1906025}

\email{casazzapeter40@gmail.com}
\email{ipcgy9@mail.missouri.edu}
\email{tinmizzou@gmail.com}
\begin{abstract} 
Let $X=\{x_i\}_{i=1}^m$ be a set of unit vectors in $\RR^n$. The coherence of $X$ is $\coh(X):=\max_{i\not=j}|\langle x_i, x_j\rangle|$. 
A vector $x\in X$ is said to be isolable if there are no unit vectors $x'$ arbitrarily close to $x$ such that $|\langle x', y\rangle|<\coh(X)$ for all other vectors $y$ in $X$.
We define the {\bf core} of a Grassmannian frame $X=\{x_i\}_{i=1}^m$ in $\RR^n$ at angle $\alpha$ as a maximal subset
of $X$ which has coherence $\alpha$ and has no isolable vectors. In other words, if $Y$ is a subset of $X$, $\coh(Y)=\alpha$, and $Y$ has no isolable vectors, then $Y$ is a subset of the core.
We will show that every Grassmannian frame of $m>n$ vectors for $\RR^n$ has the
 property that each vector in the core makes angle $\alpha$ with a spanning family from the core. Consequently, the core consists of $\ge n+1$ vectors. We then develop other properties of Grassmannian frames and of the core. 
\end{abstract}
\maketitle

\section{introduction}

For $1\leq k < n$, denote by $\mathcal{G}(k, \mathbb{F}^n)$ the set of all $k$-dimensional subspaces of a finite dimensional Hilbert space $\mathbb{F}^n$ ($\mathbb{F}=\RR$ or $\CC$). For any two elements $V, W$ in $\mathcal{G}(k, \mathbb{F}^n)$ with respective orthogonal projections $P, Q$,  the {\it chordal distance} is defined by

	$$d_C(V, W)=\frac{1}{\sqrt{2}}\| P-Q\|_{HS}=\sqrt{k-\tr(PQ)},$$ where $\|\cdot\|_{HS}$ is the Hilbert-Schmidt norm.

The {\bf subspace packing problem} is the problem of finding $m$ elements in $\mathcal{G}(k, \mathbb{F}^n)$ so that the minimum distance between any two of them is as large as possible. Note that this problem is equivalent to finding $m$ orthogonal projections, $\mathcal{P}=\{P_i\}_{i=1}^m$, each of rank $k$, so that maximum pairwise ``angle", $\tr(P_iP_j)$, is minimal. 
When all subspaces have dimension one the problem is known as the {\bf line packing problem}. In this case, if we identify each subspace with a unit vector spanning the space, then the problem is equivalent to finding a set of unit vectors $X=\{x_i\}_{i=1}^m$ in $\mathbb{F}^n$ of minimal \textit{coherence}, \[\coh(X):=\max_{1\le i\not= j\le m}\absip{x_i}{x_j}.\]
\begin{definition}
Solutions to the minimal coherence problem are called {\bf $(m, n)$-Grassmannian frames} and the minimal coherence, denoted $\alpha_{m, n}$, is called the {\bf Grassmannian angle} of the frame. 
\end{definition}
For the sake of brevity the parameters $m$ and $n$ are typically omitted if they are clear from  context. Additionally, the Grassmannian angle is typically referred to simply as the ``angle" of the frame.

There has been a significant amount of attention given to Grassmannian frames because they are tremendously useful in a wide variety of  applications such as quantum information theory \cite{RBSC, Z}, coding theory \cite{HP, SH}, digital fingerprinting \cite{MQKF}, and compressive sensing \cite{BCM, BFMW, CENR}. It is shown in \cite{SH} that Grassmannian frames also impact a number of fundamental problems in several areas of mathematics including spherical codes, spherical designs, equiangular line sets, equilateral point sets, and strongly regular graphs. 

A special class of Grassmannian frames are the so-called {\it Equiangular tight frames} (ETFs). These are unit-norm, tight frames in which the inner products between any two vectors have the same magnitude. The following important result from \cite{SH} characterizes these frames.  
\begin{theorem}[Welch]\label{Welch}
	Let $X=\{x_i\}_{i=1}^m$ be a unit-norm frame for $\mathbb{F}^n$. Then 
\[\coh(X)\geq\sqrt{\dfrac{m-n}{n(m-1)}},\] and equality holds if and only if $X$ is an equiangular tight frame. 
\end{theorem}
This bound is usually referred to as the {\it Welch bound}
\cite{We}. Theorem \ref{Welch} is significant
because Grassmannian frames are very difficult to identify in general \cite{BK, FJM, MP}, and the additional structural information afforded by ETFs makes them more accessible. Examples of ETFs include the cube roots of unity (viewed as vectors in $\RR^2$)
and the vertices of the origin-centered tetrahedron (viewed as vectors in $\RR^3$). The apparent beauty of ETFs coupled with their importance as Grassmannian frames has made them the subject of active research recently; see the survey paper \cite{FM} and references therein. It is known that ETFs can exist only if $m\leq \frac{n(n+1)}{2}$ when $\mathbb{F}=\RR$, and $m\leq n^2$ when $\mathbb{F}=\CC$, which is known as the {\it Gerzon bound} \cite{LS, SH, T}. An ETF  
which achieves the Gerzon bound is called a {\it maximal ETF}. To date, the existence of maximal ETFs is an open question, not only in frame theory but also in many other branches of mathematics. In the real case, the existence of such ETFs requires the dimension $n$ to be 2, 3, or of the form: $n=(2k+1)^2-2$ for $k\in \mathbb{N}$. Currently, examples of maximal real ETFs are known only for
dimensions $n = 2, 3, 7, 23$ \cite{FM}. Unlike the real case, maximal ETFs in the complex case are conjectured to exist in every dimension. This is known as {\it Zauner's conjecture} \cite{Z} and it is a deep problem for which solutions
in only a finite number of dimensions have been found \cite{FHS, RBSC}. In quantum mechanics, such ETFs are called symmetric, informationally complete, positive operator-valued measures (SIC-POVMs). In this context maximal ETFs are foundational to the theory of quantum Bayesianism \cite{FS} and have further applications in both quantum state tomography \cite{CFS} and quantum cryptography \cite{FS1}.

When the number of vectors exceeds the Gerzon bound, the {\it orthoplex bound} \cite{CHS, JKM, R} offers an alternative way to construct best packings of lines with specific geometric characteristics. The orthoplex bound
states that the maximal magnitude among the inner products of pairs of vectors cannot be smaller than $\tfrac{1}{\sqrt{n}}$. Examples of such Grassmannian frames are maximal sets of {\it mutually unbiased bases}. 
A family of orthonormal bases $\{e_{ij}\}_{i=1,j=1}^{\ m,\ n}$ for $\mathbb{F}^n$ is {\it mutually unbiased} if for every
$1\le i \neq \ell \le m$  and every $1\le j,k\le n$ we have $|\langle e_{ij},e_{\ell k}\rangle|=\frac{1}{\sqrt{n}}$.
Other Grassmannian frames that saturate the orthoplex bound can be found in \cite{BH1, CFHT}.

\section{Background}

Hilbert space frames were defined by Duffin and Schaeffer \cite{DS} in 1952 to address some deep problems in non-harmonic Fourier series. However, the ideas of Duffin and Schaeffer did not seem to generate much general interest outside of non-harmonic Fourier series until the landmark paper of Daubechies, Grossmann, and Meyer \cite{DGM} in 1986. After this groundbreaking work, the theory of frames began to be more widely studied. To date, frames have become a powerful tool and have broad applications in many areas of mathematics, both pure and applied. This includes time-frequency analysis \cite{Gr}, wireless communications \cite{S, SH}, compressed sensing \cite{BCM, BFMW, CENR}, coding theory \cite{HP, SH}, and quantum information theory \cite{RBSC, Z}, to name a few. In this section, we will recall some basic facts about finite frame theory 
and introduce the notation used in this paper.

Throughout, for any natural number $m$, we denote by $[m]$ the set $[m]:=\{1, 2, \ldots, m\}$. We start with the definition of finite frames. 

\begin{definition}
	A sequence of vectors $X=\{x_i\}_{i=1}^m$ in a Hilbert space $\mathbb{F}^n$ is a {\bf frame} for $\mathbb{F}^n$ if there are constants $0<A\leq B<\infty$ such that 
	\begin{align*}\label{eq1.1}
	A\Vert x\Vert^2\leq \sum_{i=1}^m\vert\langle x, x_i\rangle\vert^2\leq B\Vert x\Vert^2,  \ \mbox{ for all } x\in \mathbb{F}^n.
	\end{align*}
\end{definition}

 The constants $A$ and $B$ are called the {\bf lower} and {\bf upper frame bounds}, respectively. $X$ is said to be  an {\bf $A$-tight frame} or simply a  {\bf tight frame} if $A=B$  and a {\bf Parseval frame} if $A = B = 1$. A frame is said to be {\bf equal-norm} if all frame elements have the same norm and {\bf unit-norm} if all elements have norm one. The set of all unit-norm tight frames of $m$ vectors for $\mathbb{F}^n$ is denoted by $\FUNTF(m, \mathbb{F}^n)$. It is known that if $X\in \FUNTF(m, \mathbb{F}^n)$, then the frame bound is $A=\tfrac{m}{n}$. The numbers $\{\langle x, x_i\rangle\}_{i=1}^m$ are the called the {\bf frame coefficients} of the vector $x\in \mathbb{F}^n$. It is well-known that $X$ is a frame for $\mathbb{F}^n$ if and only if $X$ spans $\mathbb{F}^n$.
 
 Let $X=\{x_i\}_{i=1}^m$ be a frame for $\mathbb{F}^n$ and let $\{e_i\}_{i=1}^m$ be the standard orthonormal basis for $\ell_2^m$. The bounded linear operator $F: \ell_2^m\to \mathbb{F}^n$ given by
 \[F\left(\sum_{i=1}^{m}a_ie_i\right)=\sum_{i=1}^{m}a_ix_i\] is called the {\bf synthesis operator} of the frame. Its adjoint operator $F^*: \mathbb{F}^n\to \ell_2^m $ is called
 the {\bf analysis operator}. It follows that
 \[F^*(x)=\sum_{i=1}^{m}\langle x, x_i\rangle e_i.\]
 The {\bf frame operator} $S: \mathbb{F}^n\to \mathbb{F}^n$ is defined by
 \[S(x):=FF^*(x)=\sum_{i=1}^{m}\langle x, x_i\rangle x_i.\] 
Note that 
\[\langle Sx, x\rangle=\sum_{i=1}^{m}\vert \langle x,x_i\rangle\vert^2 \ \mbox{ for all } x\in \mathbb{F}^n,\]
and hence, $S$ is a positive, self-adjoint, invertible operator on $\mathbb{F}^n$.  Moreover, if $X$ is an $A$-tight frame, then 
\[\langle Sx, x\rangle=\sum_{i=1}^{m}\absip{x}{x_i}^2=A\Vert x\Vert^2=A\ip{x}{x},  \ \mbox{ for all } 
x\in \mathbb{F}^n.\]
This is equivalent to its frame operator being a multiple of the identity operator, $S=AI$. In this case, we have the following useful reconstruction formula:
\[x=\frac{1}{A}\sum_{i=1}^{m}\ip{x}{x_i}x_i, \ \mbox{ for all } x\in \mathbb{F}^n.\]
In general, reconstruction of vectors in the space is achieved from the formula:
\[x=\sum_{i=1}^{m}\ip{x}{S^{-1}x_i}x_i=\sum_{i=1}^{m}\ip{x}{x_i}S^{-1}x_i.\]
A fundamental theorem in frame theory is Naimark's Theorem \cite{CG} which states that a family of vectors in $\mathbb{F}^n$ is a Parseval frame for $\mathbb{F}^n$ if and only if it is the image of an orthonormal basis under an orthogonal projection $P$ from a larger (containing) Hilbert space. It follows that the projection $(I-P)$ also yields a Parseval frame.

The main property of frames which makes them so useful in applied problems is their redundancy. That is, each vector in the space has infinitely
many representations with respect to the frame but it also has one natural
representation given by the frame coefficients. The role played by redundancy varies with specific applications. First, redundancy provides for greater design flexibility, allowing frames to be constructed which fit a particular problem in a manner not possible by a set of linearly independent vectors. For instance, quantum tomography requires classes of orthonormal bases with the property that the modulus of the inner products of vectors from different bases are a constant. A second example comes from speech recognition, wherein a vector must be determined by the absolute value of the frame coefficients (up to a phase factor). A second major advantage of redundancy is robustness. By spreading the information over a wider range of vectors, resilience against losses (erasures) can be achieved. Erasures are, for instance, a severe problem in wireless sensor networks when transmission losses occur or when sensors are intermittently fading out, or in modeling the brain where memory cells are dying out. A further advantage of spreading information over a wider range of
vectors is to mitigate the effects of noise in a signal.

For an introduction to frame theory we recommend \cite{C, CG, OC, HKLE, W}.

\section{The core of Grassmannian frames}
In this section we develop the central idea of this paper: the core of a Grassmannian frame. We then present some properties of the core.\\

\begin{definition}
	Let $X=\{x_i\}_{i=1}^m\subset \RR^n$ be any collection of unit vectors. For any $Y\subset X$, $x\in Y$, and $\alpha\in [0,1]$, the $\alpha$-\textbf{neighbors} of $x$ in $Y$  are defined by 
	\[x_Y^\alpha:=\{y\in Y:\absip{x}{y}=\alpha\}.\]
	If $X$ is a Grassmannian frame at angle $\alpha$ then $x_X^\alpha$ is also called the set of \textbf{packing neighbors} of $x$. \\
\end{definition}

\begin{definition} 
	Let $X$ be a finite collection of unit vectors in $\RR^n$ with coherence $\alpha$ and let $x\in X$. Then
	\begin{itemize}
		\item	$x$ is an {\bf isolated vector} of $X$ if $|\langle x,y\rangle|<\alpha$ for every other vector $y$ in $X$.
		\item $x$ is an {\bf isolable vector} of $X$ if, for every $\epsilon >0$, there exists a unit vector $x'$ such that $\|x-x'\|<\epsilon$ and $x'$ is an isolated vector of $(X\setminus\{x\})\cup \{x'\}$.
		\item $x$ is a {\bf deficient vector} of $X$ if
		\[\spn(x_X^\alpha)\not=\RR^n.\]
	\end{itemize}
\end{definition}

If $x$ is isolated then obviously it is also isolable and deficient. Less clear is the following and the idea of its proof appeared in \cite{FJM}.
\begin{lemma}\label{lem0} Let $X$ be a collection of unit vectors with coherence $\alpha>0$. Then deficient vectors of $X$ are isolable vectors of $X$.
\end{lemma}
\begin{proof}
	Suppose $x\in X$ is deficient. Then there exists a $z\in (x^\alpha_X)^\perp$ such that $\|z\|=1$ and $\ip{x}{z}\geq 0$.  For any $\epsilon>0$ let  $x'= \frac{1}{\|x+\epsilon z\|}(x+\epsilon z)$. Then for all $y\in x_X^\alpha$, we have
	\[|\langle x', y\rangle|=\left|\left\langle \frac{1}{\|x+\epsilon z\|}(x+\epsilon z), y\right\rangle\right|=\frac{|\langle x, y\rangle|}{\|x+\epsilon z\|}=\frac{\alpha}{\sqrt{1+\epsilon^2+2\epsilon\langle x, z\rangle}}<\alpha.\]
	If $Y=X\setminus(x^\alpha_X\cup\{x\})$ then it remains to show $\absip{x'}{y}<\alpha$ for all $y\in Y$. Let $\sup_{y\in Y}\absip{x}{y}=\delta<\alpha$. Then for $\epsilon < \alpha-\delta$ and $y\in Y$ we have

	\[	\absip{x'}{y}=\left|\left\langle \frac{1}{\|x+\epsilon z\|}(x+\epsilon z), y\right\rangle\right|\leq\frac{\absip{x}{y}+\epsilon}{\|x+\epsilon y\|}<\frac{\delta+\alpha - \delta}{\sqrt{1+\epsilon^2+2\epsilon\ip{x}{z}}}<\alpha.\]
	Therefore $x'$ is an isolated vector of $(X\setminus\{x\})\cup\{x'\}$. It follows that $x$ is isolable.
\end{proof}

The main unsolved problem here is:

\begin{problem}
	Does any Grassmannian frame contain isolated vectors? Equivalently, does any Grassmannian frame contain isolable vectors?
\end{problem}

We believe the answer is yes, however we did not see any such a frame among known Grassmannian frames.

\begin{remark}
	The converse of Lemma \ref{lem0} does not hold. I.e., if $x\in X$ is isolable it need not be deficient.
\end{remark}

\begin{ex} 
\vskip7pt

In $\RR^3$ let
\[ x=(0,0,1),\ y_1=(\sqrt{1-\alpha^2},0,\alpha),\ y_2=(0,\sqrt{1-\alpha^2},\alpha),\]
\[ y_3=(0,-\sqrt{1-\alpha^2},\alpha),\ \mbox{ where } \alpha\in (0,1).\]
Then $|\langle x, y_i\rangle|=\alpha$ for $i=1,2, 3$ and $\mathrm{span}\{y_i\}_{i=1}^{3}=\RR^3$, therefore $x$ is not deficient. However,
if we replace $x$ with $x'=(-\epsilon,0,\sqrt{1-\epsilon^2}), (0<\epsilon<\alpha)$, then
$|\langle x',y_i\rangle|<\alpha$ for $i=1,2,3$. Therefore $x$ is isolable.
\end{ex}

\begin{lemma} \label{lemcore1} Let $X$ be a finite set of unit vectors and $Z$ be the set of islolated vectors of $X$. Let $Y=X\setminus Z$ and denote by $I(Y)$ the set of isolable vectors of $Y$. Then $I(Y) \cup Z$ is the set of isolable vectors of $X$.
\end{lemma}
\begin{proof} We only need to show that every vector in $I(Y)$ is an isolable vector of $X$. Let $y\in I(Y)$ any $\epsilon >0$ be such that $\epsilon<\min\{\coh(X)-|\langle y, z\rangle|: z\in Z\}.$ Since $y$ is an isolable vector of $Y$, there exists a unit vector $y'$ such that $\|y-y'\|<\epsilon$ and $|\langle y', x\rangle|<\coh(Y)=\coh(X)$ for all $x\in Y\setminus \{y\}$. 

We now see that we also have $|\langle y', z\rangle|<\coh(X)$ for all $z\in Z$.
Indeed, for each $z\in Z$, we estimate
\[|\langle y', z\rangle|\leq |\langle y'-y, z\rangle|+|\langle y, z\rangle|<\epsilon +|\langle y, z\rangle|<\coh(X).\] This completes the proof.
\end{proof}

\begin{lemma} \label{lemcore2} Let $X$ be a finite set of unit vectors and $I(X)=\{x_i\}_{i=1}^{\ell}$ be the set of isolable vectors of $X$. Then there exists a set of unit vectors, $Z=\{x_i'\}_{i=1}^{\ell}$, such that every $z\in Z$, $|\langle z, x\rangle|<\coh(X)$ for all other vectors $x$ in $(X\setminus I(X)) \cup Z.$
\end{lemma}

\begin{proof} Let $\alpha=\coh(X)$. By definition of isolable vectors, we can find a unit vector $x'_1$ such that $|\langle x'_1, y\rangle|<\alpha$ for all $y\in X\setminus \{x_1\}$. Let $X_1:=(X\setminus \{x_1\})\cup \{x_1'\}$ and let $\epsilon >0$ be such that $\epsilon < \alpha - |\langle x_2, x_1'\rangle|$. Since $x_2$ is an isolable vector of $X$, there exists a unit vector $x_2'$ such that 
\[\|x_2-x_2'\|<\epsilon \mbox{ and } |\langle x_2', y\rangle|<\alpha \mbox{ for all } y\in X\setminus \{x_2\}.\] We now see that we also have $|\langle x_2', x_1'\rangle|<\alpha$.
Indeed, we estimate
\[|\langle x_2', x_1'\rangle|\leq |\langle x'_2-x_2, x_1'\rangle|+|\langle x_2, x_1'\rangle|<\epsilon +|\langle x_2, x_1'\rangle|<\alpha.\] Thus, we have shown that there exists a unit vector $x_2'$ such that $|\langle x_2', y\rangle|<\alpha$ for all $y\in X_1\setminus \{x_2\}$.

Now let $X_2:=(X_1\setminus \{x_2\})\cup \{x_2'\}$ and let $\epsilon > 0$ be such that $$\epsilon <\min\{\alpha - |\langle x_3, x_1'\rangle|, \alpha - |\langle x_3, x_2'\rangle|\}.$$
Then there exists a unit vector $x_3'$ such that 
\[|\langle x_3', y\rangle|< \alpha \mbox{ for all } y\in X_3:=X_2\setminus \{x_3\}.\]

Continuing this process, eventually we have a sequence of $\ell$ unit vectors, $\{x_i'\}_{i=1}^\ell$, as in the claim.
\end{proof}

To define the core of a Grassmannian frame for $\RR^n$ at angle $\alpha>0$ we first need to see that the frame must contain at least $n+1$ non-isolable vectors.

\begin{theorem} \label{thmcore1} Let $X$ be a Grassmannian frame for $\RR^n$ at angle $\alpha$ and let $I(X)$ be the set of all isolable vectors of $X$. Then there exists a set of unit vectors $Z$ of size $|I(X)|$ such that $X':=(X\setminus I(X)) \cup Z$ is a Grassmannian frame for $\RR^n$ at angle $\alpha$, and every $z\in Z$ is an isolated vector of $X'$. Moreover, if $\alpha>0$, then $X\setminus I(X)$ contains at least $n+1$ (non-isolable) vectors.
\end{theorem}
\begin{proof}  Let $Z$ be the set as in Lemma \ref{lemcore2}. Since $X'=(X\setminus I(X))\cup Z$ has the same size as $X$, we must have $\coh(X')\ge \alpha=\coh(X)$. Note that every $z\in Z$, $|\langle z, x\rangle|<\alpha\leq \coh(X')$ for all other vectors in $X'$. So $Z$ is a set of isolated vectors of $X'$. Now let $x,y\in X'$ be such that $\absip{x}{y}=\coh(X')$. Since every $z\in Z$ is an isolated vector of $X'$, it follows that $x,y\in X\setminus I(X)$. Therefore $\coh(X')=\coh( X\setminus I(X))\leq\coh(X)$. Thus $\coh(X')=\coh(X)$ and $X'$ is Grassmannian. 

To see the last claim, we proceed by way of contradiction, so we assume that $|X\setminus I(X)|\le n$. Hence, every $x$ in $X\setminus I(X)$ is a deficient vector of this set. Because $\alpha>0$, by Lemma \ref{lem0}, every $x$ in $X\setminus I(X)$ is an isolable vector of $X\setminus I(X)$. By Lemma \ref{lemcore1}, every vector in $X'$ is an isolable vector of $X'$. But then by Lemma $\ref{lemcore2}$, there exists a set of unit vectors of the same size of $X'$ whose coherence is strictly less than $\alpha$, a contradiction. 
\end{proof}

\begin{theorem}\label{thmcore2} Let $X=\{x_i\}_{i=1}^m$ be a Grassmannian frame for $\RR^n$ at angle $\alpha$ and $I(X)$ be the set of all isolable vectors of $X$. Let $Y_1:=X\setminus I(X)$ and for each $k\geq 2$, we define a set $Y_k:=Y_{k-1}\setminus I(Y_{k-1})$, where $I(Y_{k-1})$ is the set of all isolable vectors of $Y_{k-1}$. Then 
$\cap_{k=1}^\infty Y_k \not=\emptyset$. 
\end{theorem}

\begin{proof} By Theorem \ref{thmcore1} and Lemma \ref{lemcore1}, for each $k \geq 1$, there exists a set of unit vectors $Z_k$ such that $X_k:= Y_k\cup Z_k$ is a Grassmannian frame of $m$ vectors for $\RR^n$ at angle $\alpha$, and every vector in $Z_k$ is an islolated vector of $X_k$. Since $Y_k\not=\emptyset$ and $Y_{k+1}\subset Y_k$ for all $k$, it follows that there is an index $p$ such that $Y_k=Y_p$ for all $k\geq p$. The conclusion follows. 
\end{proof}

\begin{definition} Let $X$ be a Grassmannian frame for $\RR^n$ at angle $\alpha$, and 
let $\{Y_k\}_{k\in \NN}$ be the sequence of sets defined as in Theorem \ref{thmcore2}. Then the set $\cap_{k=1}^\infty Y_k$ is called the {\bf core} of $X$ and is denoted by $C(X)$.  
\end{definition}

The following is an important property of the core.
\begin{theorem}\label{thm2} 
	Let $X$ be a Grassmannian frame for $\RR^n$ at angle $\alpha$ and let $C(X)$ be its core. Then the following hold:
	\begin{enumerate}
		\item If $\alpha=0$, then $C(X)=X$.
                  \item $C(X)$ has no isolable vectors. In particular, if $X$ is not an orthonormal basis, then every $x\in C(X)$, $\spn(x^\alpha_{C(X)}) =\RR^n$. Consequently, $|C(X)|\geq n+1$.  Moreover, if $Y$ is a subset of $X$ which has no isolable vectors, then $Y\subset C(X)$. Moreover, if $Y$ is a subset of $X$ which has no isolable vectors and $\coh(Y)=\alpha$, then $Y\subset C(X)$
	\end{enumerate}
\end{theorem}

\begin{proof} 
	(1): If $\alpha=0$ then $X$ is an orthonormal basis and by definition of the core, $C(X)=X$.


(2): By definition of the core, it has no isolable vectors. The spanning property of the core follows from Lemma \ref{lem0}. Now suppose that $Y$ is a subset of $X$, $\coh(Y)=\alpha$, and $Y$ has no isolable vectors. Let $Y_k$ be the set defined as in Theorem \ref{thmcore2}, $k\geq 1$. We will show that $Y\subset Y_k$ for all $k$, and so $Y\subset C(X)$. 

To see this, note that if $A$ is any set that contains $Y$ and $I(A)$ is the set of all isolable vectors of $A$, then $Y\subset A\setminus I(A)$. This is because if $Y\cap I(A)\not=\emptyset$, then $Y$ contains an isolable vector of $A$, which is also an isolable vector of $Y$, a contradiction. Now by definition of the sets $Y_k$, we have that $Y\subset Y_k$ for all $k$. This completes the proof.
\end{proof}

\begin{theorem} 
	If $X$ is an ETF, then $C(X)=X$.
\end{theorem}
\begin{proof} Let $\alpha$ be the angle of $X$. The conclusion is obvious if $X$ is an orthonormal basis. Now let us consider the case where $\alpha>0$. We proceed by way of contradiction so we assume that $C(X)\not=X$. Thus, $X$ must have an isolable vector. So we can replace this vector by another unit vector whose inner procduct with all other vectors of $X$, in magnitude, is strictly less than $\alpha$ and still have a Grassmannian frame for $\RR^n$. But this is impossible since by Theorem \ref{Welch}, any Grassmannian frame at angle $\alpha$ must be an ETF as well.
\end{proof}

\section{Some other properties of Grassmannian frames}
In this section, we will present some other properties of Grassmannian frames. We start with a lemma.

\begin{lemma}\label{lem2}
	Let $X=\{x_i\}_{i=1}^m$ be a unit-norm, tight frame for $\mathbb{R}^n$ with coherence $\alpha$. Then the following are equivalent.
	\begin{enumerate}
		\item $x_X^\alpha = X\setminus\{x\}$ for some $x\in X$.
		\item $X$ is an ETF.
	\end{enumerate}
\end{lemma}
\begin{proof}
	We will only prove $(1)$ implies $(2)$, as the converse is trivial.	Suppose $x_X^\alpha=X\setminus\{x\}$ for some $x\in X$. Then
	\[1+(m-1)\alpha^2=\sum_{i=1}^m\absip{x}{x_i}^2=A\|x\|^2=\frac{m}{n}.\] Hence $\alpha$ equals the Welch bound, $\sqrt{\frac{m-n}{n(m-1)}}.$ The conclusion follows.
\end{proof}

\begin{proposition} \label{Prop_tight Grass} Let $X=\{x_i\}_{i=1}^{m}$ be a tight Grassmannian frame for $\RR^n$ at angle $\alpha$. If $X$ is not an ETF, then the following hold:
\begin{enumerate}
\item For every $x\in X$, we have $|x_X^\alpha|\leq m-2$.
\item If $m$ is odd, then in addition to (1) there exists $x\in X$ such that $|x_X^\alpha|\leq m-3$.
\end{enumerate}
\end{proposition}
\begin{proof} (1):  $X$ is not an ETF, therefore (1) follows from Lemma \ref{lem2}.

(2): Proceeding by contradiction, suppose  $|x_X^\alpha|= m-2$ for every $x\in X$. Then each row of the Gram matrix of $X$ has exactly one off-diagonal entry different from $\pm \alpha$. Since the Gram matrix is symmetric, there must be an even number of such entries. But $m$ is odd, a contradiction. The result follows.
\end{proof}

Given a Grassmannian frame $X$, we have seen the spanning property of vectors in the core. The following gives another spanning property of Grassmannian frames.

\begin{theorem}\label{thm-span}
	Let $X=\{x_i\}_{i=1}^m$ be a Grassmannian frame for $\mathbb{R}^n$, $(m>n)$. Then for any $j\in [m]$, the set
	$\{x_i: i\not= j\}$ spans $\mathbb{R}^n$.
\end{theorem}

\begin{proof} Suppose by way of contradiction that there exists $j\in [m]$ such that $\spn\{x_i\}_{i\neq j}\neq \RR^n$. Let $z$ be a unit vector which is orthogonal to $\spn\{x_i\}_{i\neq j}$. Note that the set $\{x_i\}_{i\not=j}\cup \{z\}$ is also a Grassmannian frame. But the core of this frame is empty, a contradiction.
\end{proof}

\begin{remark}  The conclusion in Theorem \ref{thm-span} is the best possible since there are examples of Grassmannian frames which no longer span if 2 vectors are removed. One such example is a Grassmannian frame of 3 vectors in $\RR^2$. A less trivial example is the following:
\end{remark}
\begin{ex} Let $X$ be a frame for $\RR^4$ where the frame vectors are the columns of the matrix:
\[ \frac{1}{\sqrt{3}}\begin{bmatrix}
1&1&1&1&1&1\\
\sqrt{2}&-\sqrt{2}&0&0&0&0\\
0&0&\sqrt{2}&-\sqrt{2}&0&0\\
0&0&0&0&\sqrt{2}&-\sqrt{2}
\end{bmatrix}\]
This is an equiangular frame at angle $1/3$ for $\RR^4$. It is known that the Grassmannian angle of $6$ vectors in $\RR^4$ is $1/3$ \cite{BC}. So this is an equiangular Grassmannian frame for $\RR^4$. It is clear that we cannot drop the first two vectors leaving a spanning set for the space.\\
\end{ex}

For a given pair of positive integers, $(m, n)$, finding the Grassmannian angle, $\alpha_{m,n}$, is an extremely difficult problem. In a few cases, such as $m=n$ or $m=n+1$, calculating $\alpha_{m,n}$ is trivial or relatively straightforward. However, even for small values of $m$ and $n$, calculating $\alpha_{m,n}$ explicitly can be impossible \cite{FJM}. For this reason, most known Grassmannian frames are constructed in such a way that the coherence agrees with established bounds such as the Welch bound, or the orthoplex bound. Recently, exact values of Grassmannian angles were computed in  \cite{BC}:
\begin{itemize} 
	\item  If $n\equiv -2$ (mod 3) and $m=n+2$, then $\alpha_{m,n}=\frac{3}{2n+1}$.
	\item If $n\equiv -3$ (mod 6) and $m=n+3$, then $\alpha_{m,n}=\frac{6}{(\sqrt{5}+1)n+3(\sqrt{5}-1)}$.
	\item If $n\equiv -7$ (mod 28) and $m=n+7$, then $\alpha_{m,n}=\frac{14}{5n+21}$.
	\item If $n\equiv -23$ (mod 276) and $m=n+23$, then $\alpha_{m,n} = \frac{69}{14n+253}$.
\end{itemize}

We now give comparision of the Grassmmanian angles for different pairs $(m, n)$. 
\begin{proposition} For any natural numbers $m> n\geq 2$, the following inequalities hold:
	\begin{enumerate}
		\item $\alpha_{m,n}\leq \alpha_{m+1, n}$.
                   \item $\alpha_{m+1, n+1}<\alpha_{m, n}$.	
		\item $\alpha_{m,n+1}< \alpha_{m,n}$.	
	\end{enumerate}
\end{proposition}
\begin{proof} 
	(1): If $X$ is a Grassmannian frame of $m+1$ vectors for $\RR^n$, then dropping one vector of $X$ yields a set of $m$ vectors in $\RR^n$ with coherence less than or equal to $\alpha_{m+1, n}$. This implies that $\alpha_{m, n}\le \alpha_{m+1, n}$.

(2): Let $X$ be a Grassmannian frame of $m$ vectors for $\RR^n$ and consider the frame $X'=X\cup\{e_{n+1}\}\subset \RR^{n+1}$, where $e_{n+1}$ is the canonical basis vector. 
	By Theorem \ref{thm-span}, $X'$ is not Grassmannian for $\RR^{n+1}$. Since $\coh(X')=\alpha_{m, n}$, it follows that $\alpha_{m+1,n+1}<\alpha_{m,n}$.

	(3): This is a direct consequence of (1) and (2) sicce $\alpha_{m,n+1}\leq \alpha_{m+1,n+1}<\alpha_{m, n}.$	
\end{proof}

\begin{remark} (1) is still valid if $m=n$. Moreover, it is possible for equality to hold in (1). For example, the Grassmannian angles for both 5 and 6 vectors in $\RR^3$ are equal \cite{BK}.
\end{remark}

\begin{theorem}
	Let $X=\{x_i\}_{i=1}^m$ be a unit-norm frame for $\RR^n$ and $\lambda$ be the largest eigenvalue of its frame operator with multiplicity $k$. Then there is a set of unit vectors $Y=\{y_i\}_{i=1}^m$ in $\RR^{m-k}$ such that 
	\[\langle y_i, y_j\rangle =\frac{1}{1-\lambda} \langle x_i, x_j\rangle, \  \mbox{ for all } i\not=j.\]
	In particular, if $X$ is a Grassmannian frame at angle $\alpha$, then 
	\[\coh(Y) = \frac{\alpha}{\lambda -1}\mbox{ and } \alpha \geq (\lambda -1)\sqrt{\frac{k}{(m-k)(m-1)}}.\]\end{theorem}

\begin{proof}
	A standard result in frame theory says we can add vectors $\{z_i\}_{i=1}^{n-k}$ to $\{x_i\}_{i=1}^m$ so that together they form a
	$\lambda$-tight frame.  Multiplying the vectors by $\frac{1}{\sqrt{\lambda}}$ we get a Parseval frame $\{x_i'\}_{i=1}^{m+n-k}$, where
	\[ x_i'=\frac{1}{\sqrt{\lambda}}x_i\mbox{ for }i\in [m], \mbox{ and }\] 
	\[ x_i'=\frac{1}{\sqrt{\lambda}}z_i\mbox{ for }i=m+1,m+2,\ldots,m+n-k.\]
	So there is an orthonormal basis $\{e_i\}_{i=1}^{m+n-k}$ for $\ell_2^{m+n-k}$
	so that the projection onto the span of our Parseval frame satisfies:
	\[ Pe_i=x_i' \mbox{ for all }i\in [m].\]
	Note that $\{(I-P)e_i\}_{i=1}^{m+n-k}$ is a Parseval frame for $\RR^{m-k}$, and 
	\[\|(I-P)e_i\|^2 = 1-\|Pe_i\|^2 = 1 -\frac{1}{\lambda} \mbox{ for all }i\in [m].\]
	Now let \[ y_i=\frac{1}{\sqrt{1-\frac{1}{\lambda}}}(I-P)e_i, \mbox{ for } i\in [m]\]
	we have that $ \|y_i\|=1$ and 
\begin{align*}\langle y_i,y_j\rangle&=\frac{1}{1-\frac{1}{\lambda}}\langle (I-P)e_i,(I-P)e_j\rangle\\
&=\frac{-1}{1-\frac{1}{\lambda}}\langle Pe_i,Pe_j\rangle\\
&=\frac{-1}{1-\frac{1}{\lambda}}\frac{1}{\lambda}\langle x_i, x_j\rangle\\
&=\frac{1}{1-\lambda} \langle x_i, x_j\rangle.
\end{align*}
	for all $i, j \in [m], i\not=j$.
	
	Now assume that $X$ is a Grassmannian frame at angle $\alpha$, then 
	\[\max_{i\not=j}|\langle y_i, y_j\rangle|=\frac{1}{\lambda -1}\max_{i\not=j}|\langle x_i, x_j\rangle|=\frac{\alpha}{\lambda-1}.\]
	Moreover, by Theorem \ref{Welch} we have that
	\[\frac{\alpha}{\lambda-1}\geq \sqrt{\frac{m-(m-k)}{(m-k)(m-1)}}\] or \[\alpha\geq (\lambda-1)\sqrt{\frac{k}{(m-k)(m-1)}}.\]
The proof is complete.
\end{proof}

\begin{theorem}
Let $X=\{x_i\}_{i=1}^m$ be a Grassmannian frame for $\RR^n$ at angle $\alpha$ and let $\{e_i\}_{i=1}^n$ be the eigenvectors of the frame operator with
respective eigenvalues $\lambda_1\ge \lambda_2 \ge \cdots \ge \lambda_n$. Assume $\lambda_1$ is the largest eigenvalue among all $(m, n)$-Grassmannian frames. For each $x\in X$, let
\[ F_x=\{x\}\cup x_X^\alpha.\]
Then, $e_1\in \spn{F_x}$ for all $x\in X$.
\end{theorem}

\begin{proof}
Suppose by way of contradiction that there exists $x\in X$ such that $e_1\notin F_{x}$. Choose $\|z\|=1$ with $z\perp \spn F_x$ and $\langle z,e_1\rangle \not= 0$. Assume $x=ae_1+u$ with $u\perp e_1$. 
By replacing z with $-z$ and $x$ with $-x$ if necessary, we may assume 
\[ \langle z,e_1\rangle >0 \mbox{ and } \langle x,e_1\rangle \ge 0.\]
By Lemma \ref{lem0}, we can find a unit vector $x'$ such that $|\langle x',  y\rangle|<\alpha$ for all $y\in X\setminus \{x\}$, where $x'=\frac{x+\epsilon z}{\|x+\epsilon z\|}$ for some $\epsilon >0$. We choose $\epsilon$ small enough so that we also have $a\epsilon < 2\langle z, e_1\rangle$. Let $Y=(X\setminus \{x\}) \cup \{x'\}$, then $Y$ is also a Grassmannian frame at angle $\alpha$.

Now we can check that if $a\epsilon < 2\langle z, e_1\rangle$, then
\[|\langle x', e_1\rangle|^2=\frac{(\epsilon\langle z, e_1\rangle+a)^2}{1+\epsilon^2}>a^2.\]
Therefore,
\[ \sum_{y\in  X, y\not=x}|\langle y, e_1\rangle|^2+|\langle x',e_1\rangle|^2 >  \sum_{y\in  X, y\not=x}|\langle y, e_1\rangle|^2+a^2=\sum_{x\in X}|\langle x,e_1\rangle|^2=\lambda_1^2.
\]
It follows that the largest eigenvalue of the new frame is strictly greater than the largest eigenvalue of the old frame.
This contradicts our assumption.

\end{proof}

\section{1-Grassmannian frames}

A finite frame which is unit-norm and tight is called a FUNTF.
In addition to Grassmannian frames, one can consider the class of
frames which minimize coherence over the space of FUNTFs. Solutions to this problem are called  1-{\it Grassmannian frames} \cite{CH}. It is known that these two types of frames coincide in many settings, but not all. One noteworthy advantage of working with 1-Grassmannian frames is that their optimality properties are preserved under the Naimark complement.

Let $\mu_{m, n}$ be the coherence of a $1$-Grassmanian frame of $m$ vectors for $\RR^n$. By definition, we have that $\mu_{m, n}\geq \alpha_{m, n}$. Similar to $\alpha_{m, n}$, the exact value of $\mu_{m, n}$ is generally unknown and difficult to compute. The value of $\mu_{m,n}$ is derived below for $m=n+2$; however, we first require a theorem from \cite{CH}.

\begin{theorem}\label{thm_NM} If a $1$-Grassmannian frame in $ \FUNTF(m, \mathbb{F}^n)$ has coherence $\mu_{m, n}$, then a $1$-Grassmanian frame in $\FUNTF(m, \mathbb{F}^{m-n})$ exists, and its coherence is $\mu_{m, m-n}=\frac{n}{m-n}\mu_{m, n}$.
\end{theorem}

\begin{theorem} For any natural number $n\ge 2$, we have $\mu_{n+2, n}=\frac{2}{n}\cos\frac{\pi}{n+2}$.
\end{theorem}
\begin{proof}
It is known \cite{BK} that the frame $\Phi=\{\varphi_k\}_{k=1}^{n+2}$, where
\[\varphi_k=\left(\cos\frac{k\pi}{n+2}, \sin\frac{k\pi}{n+2}\right)\] is a tight Grassmannian frame (and hence, a 1-Grassmannian frame) for $\RR^2$ with coherence $\alpha =\cos\frac{\pi}{n+2}$. 
The conclusion then follows by Theorem \ref{thm_NM}.
\end{proof}

Since $\mu_{n+2, n}> \sqrt{\frac{2}{n(n+1)}}$ (the Welch bound for the case $m=n+2$), we recover the following known fact.
\begin{corollary} 
	There exists no ETF of $n+2$ vectors for $\RR^n$ if $n\geq 2$.
\end{corollary}

\begin{theorem} For any $2\leq n\leq m$, the 1-Grassmannian constant $\mu_{m,n}$ satisfies $\mu_{2m, 2n}\leq \mu_{m, n}$.
\end{theorem}

\begin{proof} Suppose $X=\{x_i\}_{i=1}^m$ is a 1-Grassmannian frame for $\RR^n$ at angle $\mu_{m, n}$ and with bound $A$.
For each $i\in [m]$, we define a unit-norm frame $X'=\{x_i^+\}_{i=1}^m\cup\{x_i^-\}_{i=1}^m$ for $\RR^{2n}$ by
\[ x_i^+=\frac{1}{\sqrt{2}}(x_i,x_i)\mbox{ and } x_i^-=\frac{1}{\sqrt{2}}(x_i,-x_i).\]
We need to check that $X'$ is a tight frame for $\RR^{2n}$ with $\coh(X')\leq\mu_{m,n}$.

They are clearly unit-norm. We will check that they form a tight frame for $\RR^{2n}$. For any $x\in \RR^{2n}$, we write $x=(y, z)\in \RR^n\times \RR^n$ and compute
\begin{align*}
\frac{1}{2}\sum_{i=1}^m|\langle (y,z),(x_i,x_i\rangle|^2&=
\frac{1}{2}\sum_{i=1}^m|\langle y,x_i\rangle+\langle z,x_i\rangle|^2\\
&= \frac{1}{2}\sum_{i=1}^m\left [ \langle y,x_i\rangle^2 +\langle z,x_i\rangle^2+2 \langle y,x_i\rangle \langle z,x_i\rangle\right ].
\end{align*}
Similarly,
\[ \frac{1}{2}\sum_{i=1}^m|\langle (y,z),(x_i,-x_i\rangle|^2=
\frac{1}{2}\sum_{i=1}^m\left [ \langle y,x_i\rangle^2 +\langle z,x_i\rangle^2-2 \langle y,x_i\rangle \langle z,x_i\rangle\right ].\]
So adding these gives a tight frame bound of A.

Now we compute inner products.  For all $i, j\in [m], i\not=j$, 
\[ |\langle x_i^+,x_j^+\rangle|=\frac{1}{2}2|\langle x_i,x_j\rangle|\le \mu_{m, n}.\]
Also, $ |\langle x_i^-,x_j^-\rangle|\leq \alpha$. Moreover,
\[ |\langle x_i^+,x_j^-\rangle|=\frac{1}{2}\langle (x_i,x_i\rangle,(x_j,-x_j\rangle|=
\frac{1}{2}\left [ \langle x_i,x_j\rangle -  \langle x_i,x_j\rangle\right ]=0.
\] Thus, this frame has cohencece less than or equal to $\mu_{m, n}$.
\end{proof}
Note that the proof of this theorem not only shows that $\mu_{2m, 2n}\leq \mu_{m, n}$ but also gives a construction of a tight frame of $2m$ vectors for $\RR^{2n}$ with cohecence no greater than $\mu_{m, n}$.

\section{Grassmannian frames of $n+2$ vectors in $\RR^n$}
Studying in detail Grassmannian frames of $n+2$ vectors in $\RR^n$ for some small values of $n$ has been given in \cite{BK, FJM, MP}. In this section, we will give some general properties of these frames.
\begin{theorem}\label{thm3}
Let $X=\{x_i\}_{i=1}^{n+2}$ be a Grassmannian frame for $\RR^n$ at angle $\alpha$. Then one of the following must hold:
\begin{enumerate}
\item $X$ contains a subset of $n+1$ vectors which are equiangular.
\item Every vector in the frame is in the core, i.e., $C(X)=X$.
\end{enumerate}
\end{theorem}

\begin{proof} Suppose that $C(X)\not=X$. Then $|C(X)|=n+1$ by Theorem \ref{thm2} and so $X$ has exactly one isolable vector. Since $\spn(x_{C(X)}^\alpha)=\RR^n$ for all $x\in C(X)$ we must also have $\absip{x}{y}=\alpha$ for all $x,y\in C(X), x\not=y$. Thus $C(X)$ is an equiangular subset of $X$ with $n+1$ vectors.
\end{proof}

\begin{theorem} A Grassmannian frame of $n+2$ vectors in $\RR^n, n>2$ cannot be tight.
\end{theorem}
\begin{proof} We proceed by way of contradiction. Suppose $X=\{x_i\}_{i=1}^{n+2}$ is a tight Grassmannian frame for $\RR^n$ at angle $\alpha$. Note that there are no ETFs of $n+2$ vectors in $\RR^n$ and by Theorem \ref{thm2}, the core $C(X)$ has at least $n+1$ elements. We consider two cases:

Case 1: $|C(X)|=n+1$. By Theorem \ref{thm3}, $X$ contains an equiangular subset of $n+1$ vectors at angle $\alpha$. We can assume this set $\{x_i\}_{i=1}^{n+1}$. Then for any $j\in [n+1]$, we have
\[\frac{n+2}{n}=\sum_{i=1}^{n+1}|\langle x_j, x_i\rangle|^2+|\langle x_j, x_{n+2}\rangle|^2.\] It follows that $|\langle x_j, x_{n+2}\rangle|$ is the same for all $j\in [n+1]$. If we denote this constant by $\beta$, then $\beta<\alpha$ and
\[\frac{n+2}{n}=1+n\alpha^2+\beta^2.\]  But we also have
\[\frac{n+2}{n}=\sum_{i=1}^{n+2}|\langle x_{n+2}, x_i\rangle|^2=1+(n+1)\beta^2.\] Thus, $\alpha=\beta$, a contradiction.

Case 2: $|C(X)|=n+2$. Since $X$ is not an ETF, by Proposition \ref{Prop_tight Grass}, $|x^\alpha_X|=n$ for all $x\in X$. Thus, on each row of the Gram matrix of $X$, there exists exactly one number whose magnitude is strictly less than $\alpha$. Since $X$ is tight, it follows that these numbers, in magnitude, is the same. If we denote it by $\beta$, then
\[n\alpha^2+\beta^2=\frac{2}{n}.\] Moreover, by Theorem \ref{thm_NM}, we must have $\alpha=\mu_{n+2, n}=\frac{2}{n}\cos\frac{\pi}{n+2}$. But then
\[n\alpha^2=\frac{4}{n}\cos^2\frac{\pi}{n+2}>\frac{2}{n},\] if $n>2$, a contradiction. The proof is complete.
\end{proof}

\begin{remark} It is easy to find a tight Grassmannian frame of 4 vectors in $\RR^2$, for example, take the union of 2 mutually unbiased bases for $\RR^2$.
\end{remark}

\end{document}